\documentclass[11pt, reqno,oneside]{amsart}

\usepackage{amscd,amsmath}
\usepackage{mathrsfs}
\usepackage{amsfonts}
\usepackage{amssymb}
\usepackage{enumerate}
\usepackage{setspace}
\usepackage{color}
\usepackage{bm}
\usepackage{esint}
\usepackage{float}

\usepackage{tikz}
\usetikzlibrary{positioning,fadings,through}

\usepackage{anysize}

\usepackage[english]{babel}

\usepackage{tabu}

\usepackage[colorlinks=true, citecolor=blue]{hyperref}

\newtheorem{theorem}{Theorem} [section]

\newcommand{\eqn}{\begin{eqnarray}}
\newcommand{\een}{\end{eqnarray}}

\newcommand{\ZZ}{\mathbb{Z}}
\newcommand{\TT}{\mathbb{T}}
\newcommand{\pat}{\partial_t}

\definecolor{luh-dark-blue}{rgb}{0.0, 0.313, 0.608}

\def\jump#1{{[\hspace{-2pt}[#1]\hspace{-2pt}]}}

\newcommand{\bR}{\mathbb{R}}
\newcommand{\bT}{\mathbb{T}}
\newcommand{\bZ}{\mathbb{Z}}

\numberwithin{equation}{section}

\usepackage{dutchcal} 

\usepackage{geometry}
\usepackage{graphicx}

\usepackage{scalerel,stackengine}
\stackMath
\newcommand\reallywidehat[1]{%
\savestack{\tmpbox}{\stretchto{%
  \scaleto{%
    \scalerel*[\widthof{\ensuremath{#1}}]{\kern-.6pt\bigwedge\kern-.6pt}%
    {\rule[-\textheight/2]{1ex}{\textheight}}
  }{\textheight}%
}{0.5ex}}%
\stackon[1pt]{#1}{\tmpbox}%
}
\usepackage{soul}
\usepackage{mathtools}%

\def\al{\alpha}

\newcommand{\m}[1]{{\mathbcal{#1}}}
\newcommand{\red}[1]{{\textcolor{black}{#1}}}

\newcommand{\w}[1]{{\widehat{#1}}}
\newcommand{\wb}[1]{{\overline{#1}}}

\newcommand{\mS}[1]{{S^{(#1)}}}
\newcommand{\mB}[1]{{B^{(#1)}}}
\newcommand{\mP}[1]{{P^{(#1)}}}
\newcommand{\mh}[1]{{h^{(#1)}}}

\def\q{\quad}
\def\qq{\qquad}

\def\Ga{\Gamma}

\def\lm{\lambda}
\def\de{\delta}
\def\si{\sigma}
\def\ga{\gamma}
\def\b{\beta}
\def\pa{\partial}
\def\eps{\varepsilon}
\def\N{\nabla}
\def\D{\Delta}
\def\mg{\mathbcal{h}}
\def\mg{\mathbcal{g}}

\def\cm{\t{cm}}
\def\g{\t{g}}
\def\tsi{\tilde{\si}}
\def\s{\t{s}}

\newcommand{\pare}[1]{\left( #1 \right)}
\newcommand{\norm}[1]{\left\| #1 \right\|}
\newcommand{\av}[1]{\left| #1 \right|}
\newcommand{\bra}[1]{\left[ #1 \right]}
\newcommand{\set}[1]{\left\{ #1 \right\}}

\renewcommand{\t}[1]{\text{#1}}

\allowdisplaybreaks

\begin{document}

\setlength{\abovedisplayskip}{5pt}
\setlength{\belowdisplayskip}{5pt}

\setlength{\jot}{7pt}

\title{A nonlocal equation describing tumor growth}

\author{Rafael Granero-Belinch\'on}
\address{Departamento de Matem\'aticas, Estad\'istica y Computaci\'on, Universidad de Cantabria, Spain.}
\email{rafael.granero@unican.es}

\author{Martina Magliocca}
\address{Departamento de An\'alisis Matem\'atico, Universidad de Sevilla, Spain.}
\email{mmagliocca@us.es}

\date{\today}

\subjclass[2010]{}

\keywords{Tumor growth, Free boundary problem, Nonlocal PDE, Well-posedness}

\begin{abstract}
Cancer is a very complex phenomenon that involves many different scales and situations. In this paper we consider a free boundary problem describing the evolution of a tumor colony and we derive a new asymptotic model for tumor growth. We focus on the case of a single phase tumor colony taking into account chemotactic effects in an early stage where there is no necrotic inner region. Thus, our model is valid for the case of multilayer avascular tumors with very little access to both nutrients and inhibitors or the case where the amount of nutrients and inhibitors is very similar to the amount consumed by the multilayer tumor cells. Our model takes the form of a single nonlocal and nonlinear partial differential equation for the interface of the multilayer tumor colony. Our model is able to capture chemotactic and cohesion effects and also the effect of nutrients, inhibitors and vasculation of the tumor colony. Besides deriving the model, we also \red{prove} a well-posedness result.
\end{abstract}

\maketitle

\setcounter{tocdepth}{3}

{\small \tableofcontents}

\section{Introduction}
Free boundary problems for partial differential equations are a very hot research area in Mathematical Analysis nowadays. These problems are mathematically challenging and physically interesting. Moreover, their applications are really spread, from geothermal reservoirs to tumour growth, passing through weather forecasting. Many free boundary problems arise during the interaction between two fluids or between a fluid and an elastic solid, \emph{i.e.} when the dynamics of both fluids or the fluid and the elastic solid are connected. The study of fluid \& fluid and fluid \& solid interactions are classical problems in Applied Sciences and Mathematics. In this paper, we will consider tumor growth as a free boundary problem for appropriate partial differential equations.

Cancer is a very complex phenomenon that involves many different scales and situations. Cancer starts at a primary site in the body with the mutation of a small number of cells. Some time after this, proliferating cancer cells start their cell division and the tumor grows. Some of the cells situated in the inner part of the tumor die due to absence of enough nutrients such as oxygen or glucose and a central necrotic core is formed. Usually, the closest cells to this necrotic core are viable non-proliferating cells while the outer part of the carcinoma contains the proliferating cells \cite{byrne1995growth,byrne1996growth}. Without access to the vasculature, the tumor colony may grow up to the order of $10^6$ cells.  If these cancer cells manage to grow successfully they may acquire a vascular network. Once they have access to a constant source of nutrients, the tumor can grow beyond that limit and at a faster rate \cite{sutherland1971growth}.

The importance of this illness together with the complexity of the phenomenon lead to a large number of research studies by many different research groups in Biology and Medicine but also in Physics, Mathematics or Engineering. In fact, the first mathematical models related to cancer goes back to the fifties (see \cite{araujo2004history,byrne2010dissecting} for very complete reviews). However, these first attempts do not provide any mechanical insight in the \red{growth} of tumors. Due to the large number of cells involved, a continuum modelling approach seems appropriate \cite{byrne2010dissecting}.

Although cancer research has improved the available treatments and our understanding of this disease, there is still a lack of a good mechanical description of certain phenomena. Such mathematical approach to oncology has many benefits as evidenced for instance by the work of Gatenby \& Gawlinski \cite{gatenby1996reaction} (see also \cite{anderson2008integrative,byrne2010dissecting}). There, the predictions of the mathematical models where verified experimentally later. Furthermore, as pointed out in \cite{anderson2008integrative}, mathematical models could provide unexpected insight into the underlying mechanisms and generate novel hypotheses for experimentation. The purpose of this work is to study a mechanical model of tumor growth and derive a new asymptotic model that accurately describe the evolution of the colony. In particular, we will focus on the case of a single phase tumor colony where only the tumor (without healthy surrounding tissue) is considered (see figure \ref{Fig1}). The tumor is assumed to be in an early stage where there is no necrotic inner region but we allow the tumor colony to possibly have access to the vascular network. The problem that we consider is very similar to the growth of a so called \emph{multilayer tumor}, \emph{i.e.} a colony of tumor cells cultivated in laboratory \cite{kim2004three,mueller1997three}.

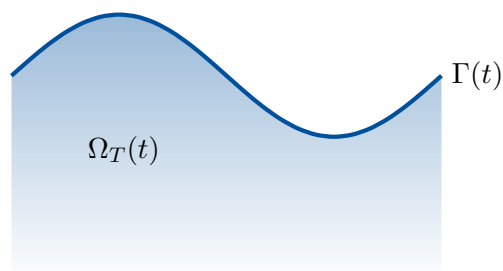
\begin{figure}[H]\label{Fig1}
\centering
\begin{tikzpicture}[domain=-pi:pi, scale=0.9]

\shade[top color=luh-dark-blue!40, bottom color=white] plot[domain=-pi:pi] (\x,{2.1-0.9*sin(\x r)}) -- plot[domain=pi:-pi] (\x,{-1});

\draw[ultra thick, smooth, variable=\x, luh-dark-blue] plot (\x,{2.1-0.9*sin(\x r)});

\node[right]  at (pi,2.1) {$\Ga(t)$};
\node at (1.4,2.5) {};
\node  at (-1.5,1) {$\Omega_T(t)$};
\end{tikzpicture}
\caption{The one-phase case with only the tumor colony and infinite depth}
\end{figure}

In that regards, when studying tumor growth there are different approaches present in the literature, each of them focusing on a specific feature of such a complex phenomenon.

First, one has the approach based on systems of nonlinear differential equations posed on a fixed domain. We could possibly call this approach the \emph{reaction-diffusion} approach. Examples of such an approach are the works \cite{bellomo2015toward,chaplain2005mathematical,gerisch2008mathematical,granero2023nonlocal,hillen2013convergence,szymanska2009mathematical,stinner2014global,tao2008global,zhigun2016global}. Typically, the system of differential equations takes the form of a coupled parabolic partial differential equation(s) and a ordinary differential equation to model haptotaxis. In general, such systems of differential equations are able to exhibit very complex and even spatio-temporal chaotic dynamics \cite{burczak2016generalized,painter2011spatio}.

Second, one can consider \emph{hybrid multiscale} and also \emph{agent based} models such as \cite{naumov2010influence,oden2016toward,rahman2017fully,rocha2018hybrid,wang2007simulating,wang2015simulating}. In this approach, partial differential equations of reaction-diffusion type are considered at the tissue level. However, at the cell level, agent based models are used to track each cell together with its state individually. In particular, these models allow to differentiate cancer cells into quiescent, proliferative, apoptotic, hypoxic, and even necrotic states. Finally, ordinary differential equations are used at the subcellular scale.

A third approach could possibly be called the \emph{phase-field} models. In this approach, the tumor microenvironment is considered as a mixture of tumor and healthy tissue. For this two-constituent medium, a function $\phi$ describes the tumor and healthy tissue volume fractions, respectively and a diffuse, typically thin, interface is considered. In this approach, the tumor is mathematically described by Cahn-Hilliard-type equations. The Cahn-Hilliard equation is a fourth order parabolic equation that was originally developed as a model phase separation and now is widely applied in tumor growth. Examples of such an approach are \cite{colli2014cahn,colli2017optimal,colli2017asymptotic,garcke2016cahn,garcke2018multiphase,lima2016selection,oden2010general,oden2016toward}.

There is however another approach which is focused on the boundary of the tumor colony. In this approach, the tumor colony is described mathematically as a \emph{free boundary problem}. These approach started with the pioneer works of Greenspan \cite{greenspan1972models,greenspan1976growth} (see also Byrne \& Chaplain \cite{byrne1996modelling,byrne1995growth,byrne1996growth}) and continues nowadays (see \cite{blanco2021modeling,conte2021modeling} and the references therein).

In particular, these early models by Greenspan study the growth of \emph{avascular}, \emph{i.e.} the carcinoma is not linked to the blood vessels, multicellular spheroids. Typically this type of tumor grows \emph{in vitro}. This situation is then described by several reaction diffusion equations for the different nutrients and inhibitors, and an ordinary differential equation for the radius of the tumor (which is assumed to be radially symmetric). This is the approach in the works of many authors such as those in \cite{cui2000analysis,cui2001analysis,friedman1999analysis,friedman2004hierarchy,friedman2007mathematical}.

Although this setting was succesful in describing certain phenomena related to cancer \red{growth} it  is also very limited. In particular, such a setting can describe how the tumor \red{grows} or how the colony diminishes, but it cannot describe observed features of certain tumors such as invasive fingers or tumor boundaries with large and variable curvature.

This paper adscribes himself to the so called free boundary approach. In this regards, we consider the problem of tumor growth as the dynamics of a free boundary. The main purpose of this paper is to derive and study a single simplified nonlinear and nonlocal partial differential equation that describe (up to certain degree of accuracy) the full free boundary dynamics in certain biological regime. A similar research program has been widely applied in many different areas in Applied Sciences. That is for instance the case of Coastal Engineering where asymptotic models of water waves are used to accurately describe the motion of the surface of the sea in the appropriate physical regime. While there are many results for free boundary problems arising when studying the motion of incompressible fluids (see for instance \cite{granero2020growth,lannes2013water,alazard2018stabilization,CCFGL,Cou-Shko} and the references therein), much less is known for free boundary problems modelling the growth of different sorts of tumors.

In this regards, the asymptotic model that we derive and study in this paper reads
\begin{align*}
\pat g(x_1,t)&=-\eta \Lambda^3 g(x_1,t)+\eps\eta\pare{\jump{H,g(x_1,t)}H g,_{111}(x_1,t)},_1\nonumber
\\
&\q -\eps\theta g(x_1,t)e^{-N_2t}\m{L},_1(x_1,0,t)+\eps \widetilde{K}^{(1)}(x_1,0,t)
+K^{(0)}(x_1,0,t).
\end{align*}
for appropriate, explicit (see below) functions $\widetilde{K}^{(1)},K^{(0)}$ and $\m{L}$. Such a model is valid up to certain explicit error (see below) in the case of an avascular tumor with very little access to both nutrients and inhibitors. Similarly, the same model is also valid in the case of vascular tumors where the amount of nutrients and inhibitors is very similar to the amount consumed by the tumor cells.

To the best of our knowledge, not only the model of tumor growth that we find is new and interesting, but also the derivation technique is different to the more heuristic derivations present in the literature. Furthermore, besides deriving a new asymptotic model for multilayer tumor growth, we also establish a well-posedness result for a specific case.

\subsection*{Notation}
We denote the one dimensional torus by $ \bT = \bR / 2\pi\bZ $. Alternatively, this domain can be thought as the interval $ [-\pi, \pi] $ with periodic boundary conditions.\\
 Given a matrix $ A \in \bR^{n\times m} $ we denote with $A_j^i $ the entry of $ A $ at row $ i $ and column $ j $, and we adopt Einstein convention for the summation of repeated indexes.\\
 We write
\[
f,_j=\pa_{x_j}f
\]
for the space derivative in the $j-$th direction. Similarly, $H$ denotes the Hilbert transform. Let $v(x_1)$ denote a $L^2$ function on $\bT$. We recall its Fourier series representation:
	\begin{equation*}
	\hat{v}(n) = \frac{1}{2\pi}\int _{\bT} v(x_1) e^{-in x_1}dx_1,
	\end{equation*}
where $ n\in\bZ $. Then we have that
\begin{equation*}
	v(x_1) =\sum_{n\in\bZ} \hat{v}(n) \ e^{in x_1}.
	\end{equation*}
Finally,
\begin{equation}\label{eq:commutator}
\jump{A,B}=AB-BA
\end{equation}
denotes the commutator.

\section{Free boundary problems describing tumor growth}
We consider the case in which the growth of cancer regions is affected by the action of certain inhibitors, whose aim is preventing the tumor expansion. These inhibitors can be a blood-borne anticancer drug or any other substance secreted by neighboring cells in response to the anomalous tissue. Then, the growth of the tumor is related to the interactions between the nutrients such as oxygen or glucose and the previously mentioned inhibitors. However, we consider that every nutrient acts via the function $\sigma$ and every inhibitor products its effect through the function $\beta$. Due to these assumptions, the system of PDE will contain two reaction diffusion equations, one for $\beta$ and one for $\sigma$.

As we said, we allow for the tumor colony to have access to the vasculature. The capillary network can supply nutrients to the tumor at a rate
$$
\delta_n(\sigma_B-\sigma),
$$
where $\sigma_B$ is the constant nutrient concentration in the vasculature, and $\delta_n$ is the constant rate of blood-tissue transfer per unit length.  We observe that the avascular case reduces then to $\delta_n=0$\\
The nutrient then diffuses in the tumor tissue with diffusion constant $D^T_n$ and it is consumed at a rate $\lambda_n \sigma$. Furthermore, the inhibitor may interact with the nutrient with a function $G_n(\sigma,\beta)$.\\
 Similarly, the inhibitor diffuses with constant $D^T_i$, it is consumed with constant $\lambda_i$ and it is supplied from the capillary network with rate $\delta_i$. \red{In the case of blood-borne anticancer drug, a blood-tissue exchange rate of the form
$$
\delta_i(\beta_B-\beta)
$$
has to be taken into account. Here, $\beta_B$ is the constant inhibitor concentration in the vasculature.}
In addition to this, the nutrient can potentially have an effect on the inhibitors via the function $G_i(\sigma,\beta)$. We can use the same reasoning with respect to the inhibitor $\beta$. Then, writing $\Omega_T(t)$ for the tumor region, these reaction diffusion equations take the generic form (see for instance \cite{byrne1995growth,lowengrub2009nonlinear})
\begin{align*}
\frac{\partial \sigma}{\partial t}-D_n\Delta \sigma&=\delta_n\left(\sigma_B-\sigma\right)-\lambda_n \sigma + G_n(\sigma,\beta)&&\text{in }\Omega_T(t)\times(0,T),\\
\frac{\partial \beta}{\partial t}-D_i\Delta \beta&=\delta_i\left(\beta_B-\beta\right)-\lambda_i \beta +G_i(\sigma,\beta)&&\text{in }\Omega_T(t)\times(0,T).
\end{align*}
The sources $G_i$ depend on the hypotheses of the model \cite{byrne1995growth}. In general, it is assumed that the inhibitor affects both cell proliferation and the nutrient concentration.\\
\red{Following \cite{byrne1995growth}, we are going to choose
\[
G_n(\si,\b)=-\ga_n \b,
\]
and
\[
G_i(\si,\b)=0,
\]
\textit{i.e.} the inhibitor affects the nutrient concentration with some constant rate $\ga_n$, but not the cell proliferation rate.\\
For more examples of choices of $G_n$ and $G_i$, we refer to \cite{byrne1995growth}.
}

The cells are assumed to move according to \emph{pressure} gradients created by the birth and death of cells  \cite{greenspan1976growth}. The introduction of this oncotic pressure function, $p$, make a conceptual change from a mechanical viewpoint. As a consequence of this hypothesis, the velocity vector of the tumor, $ u $, is assumed to follow Darcy's law
\begin{align*}
 u &=-\nabla p+\chi\nabla \sigma&&\text{in }\Omega_T(t)\times(0,T),
\end{align*}
where
$$
\chi\nabla \sigma\q \t{with}\q \chi\geq0
$$ is a chemotaxis term modelling the motion towards high concentration of nutrients. Being $\mu$ \red{a constant threshold} reflecting the expansion by mitosis and $\tilde{\sigma}$ being a threshold concentration, the local rate of volume change is given by
\begin{align*}
\nabla\cdot  u &=\mu(\sigma-\tilde{\sigma}-\tau\beta)&&\text{in }\Omega_T(t)\times(0,T),
\end{align*}
where the term $\tau\beta$ is the death rate of tumor cells due to the effect of the inhibitor. The role played by this threshold concentration $\tilde{\sigma}$ can be explained as follows. On the one hand, if the amount of nutrients is below this value in a certain region, the cells die and the tumor colony diminishes locally. On the other hand, if the amount of nutrients is above this value in certain area of the tumor tissue, the cells replicate themselves and the tumor colony \red{grows} locally.

There are a number of boundary conditions that need to be imposed for this system. First, let us denote the tumor surface by $\Gamma(t)$. We assume that the tumor colony has boundary given by
$$
\Gamma(t)=(x_1,h(x_1,t)),
$$
for certain function $h$. We write $\mathcal{K}$ for the curvature of $\Gamma(t)$, \emph{i.e.}
$$
\mathcal{K}=\frac{\partial_{x_1}^2h}{(1+(\partial_{x_1}h)^2)^{3/2}}.
$$
Then, although there are other available models \cite{byrne1996modelling}, in this work we consider that the compactness of the tumor is mantained by surface tension via a Young-Laplace equation \cite{greenspan1976growth}
\begin{align*}
p&=-\nu\mathcal{K}&&\text{on }\Gamma(t)\times(0,T).
\end{align*}
Then, as the surface of the tumor colony is transported by the cell velocity, we have that
\begin{align*}
\frac{\partial h}{\partial t}&= u \cdot (-\partial_{x_1}h(x_1,t),1)&&\text{on }\Gamma(t)\times(0,T).
\end{align*}

The appropriate boundary conditions are
\begin{align*}
\si=\si_D,\q \beta&=\beta_D&&\text{on }\Gamma(t)\times(0,T).
\end{align*}

As a consequence, under the assumption that the inhibitor acts on the nutrient concentration and the cell proliferation, the free boundary problem can be written as follows
\begin{align*}
\pat\si-D_n\Delta \sigma&=\delta_n\left(\sigma_B-\sigma\right)-\lambda_n \sigma -\gamma_n \beta && \t{in }\Omega_T(t)\times(0,T),\\
\pat\beta-D_i\Delta \beta&=\delta_i\left(\beta_B-\beta\right)-\lambda_i \beta && \t{in }\Omega_T(t)\times(0,T),\\
\N\cdot u &=\mu(\si-\tilde{\si}-\tau\beta)&& \t{in }\Omega_T(t)\times(0,T),\\
 u &=-\nabla p+\chi\nabla \si && \t{in }\Omega_T(t)\times(0,T),\\
p&= -\nu\frac{h,_{11}}{(1+h,_1^2)^{3/2}}&& \t{on }\Gamma(t)\times(0,T),\\
\pat h&= u \cdot (h,_1,1)&& \t{on }\Gamma(t)\times(0,T),\\
\si&=\si_D&& \t{on }\Gamma(t)\times(0,T),\\
\beta&=\beta_D&& \t{on }\Gamma(t)\times(0,T),
\end{align*}
with the initial data
\begin{align*}
\si(x,0)&=\si_0(x)&&\t{in }\Omega(0),\\
\beta(x,0)&=\beta_0(x)&&\t{in }\Omega(0),\\
h(x_1,0)&=\mg(x_1)&&\t{on }\Gamma(0),
\end{align*}
with
$$
\Omega_T(t)=\set{ x \in\mathbb{R}^2, x_1\in L\mathbb{S}^1,\;-\infty<x_2<h(x_1,t)},
$$
where $L$ is the typical length of the tumor colony and $L\mathbb{S}^1$ denotes the circle with length 2$L$ (or equivalently the interval $[-L\pi,L\pi]$ with periodic boundary conditions). Similarly, the boundary $\Gamma(t)$ of this domain is
$$
\Gamma(t)=\set{ x \in\mathbb{R}^2, x_1\in L\mathbb{S}^1,\;x_2=h(x_1,t)}.
$$
In the previous free boundary problem, the units of the quantities and parameters are as follows.
\begin{align*}
\t{Concentrations }\si,\,\b,\,\si_B,\,\b_B,\,\si_D,\,\b_D,\,\tilde{\si},&:\q \g\cdot\cm^{-2}\\
\t{Diffusion coefficients } D_i,\, D_n \t{ and pressure } p&:\q   \cm^{2}\cdot\s^{-1}\\
\t{Rates }\de_i,\,\de_n,\,\lm_i,\,\lm_n,\,\ga_n&:\q \s^{-1}\\
\t{Velocities } u ,\, \frac{\pa h}{\pa t}&:\q \cm\cdot\s^{-1}\\
\t{Curvature }\mathcal{K}&:\q\cm^{-1}\\
\t{Parameter }\nu&:\q\cm^3\cdot\s^{-1}\\
\t{Parameter }\chi&: \q \g^{-1}\cdot\cm^4\cdot\s^{-1}\\
\t{Parameter }\mu&: \q \g^{-1}\cdot\cm^2\cdot\s^{-1}
\end{align*}
With these units in mind, we consider the dimensionless variables
\[
\wb{x}=\pare{\frac{x_1}{L},\frac{x_2}{L}}
,\qq \wb{t}=\frac{D_n}{LH}t,
\]
and functions
\begin{equation*}
\wb{\si}(\wb{x},\wb{t})=\frac{\si(x,t)-\si_D}{\tsi},\qq\wb{ u }(\wb{x},\wb{t})=\frac{L}{D_n} u (x,t),\qq \wb{p}(\wb{x},\wb{t})=\frac{p(x,t)}{D_n},
\end{equation*}
\begin{equation*}
\wb{h}(\wb{x},\wb{t})=\frac{h(x,t)}{H},\qq \wb{\b}(\wb{x},\wb{t})=\frac{\b(x,t)-\b_D}{\tsi},
\end{equation*}
being $L$ and $H$ typical lengths associated to the tumor colony. Furthermore, we define the dimensionless parameters
\begin{equation}\label{eq:eps-al}
\eps=\frac{H}{L},\q \alpha=\frac{D_i}{D_n},
\end{equation}
\begin{equation}\label{eq:N123}
N_1=LH\frac{\de_i+\lm_i}{D_n},\q N_2=LH\frac{\de_n+\lm_n}{D_n},\q N_3=\frac{\ga_n  LH}{D_n},
\end{equation}
\begin{equation}\label{eq:M12}
M_1=\frac{LH}{D_n\tsi}\bra{-(\de_i+\lm_i)\b_D+\de_i\b_B}
,\q  M_2=\frac{LH}{D_n\tsi}
\bra{
-\si_D(\de_n+\lm_n)+\si_B\de_n-\ga_n\b_D
},
\end{equation}
\begin{equation}\label{eq:theta-rho-om-eta}
\theta=\frac{\chi\tsi}{D_n},\q \rho=\frac{\mu\tsi L^2}{D_n},\q \omega=\frac{\mu L^2}{D_n}\pare{\si_D-\tau\b_D-\tsi}
,\q \eta=\frac{\nu H}{L^2D_n}.
\end{equation}
Then, dropping the $\wb{\cdot}$ from the notation and substituting $ u $ by its value, the problem in dimensionless form reads as follows
\begin{subequations}\label{eq:pb-adim-LL_inf-1phase}
\begin{align}
\pat \beta-\eps\alpha\D \b&=-N_1\b+M_1&&\text{ in }\Omega_T(t)\times(0,T),\label{eq:bT-adim-LL_inf-1phase}\\
\pat \sigma- \eps\D \si&=-N_2\sigma-N_3\b+M_2&&\text{ in }\Omega_T(t)\times(0,T),\label{eq:siT-adim-LL_inf-1phase}\\
-\D p+\theta\D \si&=\rho\pare{\si-\tau \b}+\omega&&\text{ in }\Omega_T(t)\times(0,T),\label{eq:divu-adim-LL_inf-1phase}\\
p&= -\eta\frac{\partial_{x_1}^2h}{(1+(\eps\partial_{x_1}h)^2)^{3/2}}&& \t{ on }\Gamma(t)\times(0,T),\label{eq:p-adim-LL_inf-1phase}\\
\pat h&=\left(-\N p+\theta \N\si\right)\cdot (-\partial_{x_1}h,1)&& \t{ on }\Gamma(t)\times(0,T),\label{eq:dth-adim-LL_inf-1phase}\\
\si &=0&&\text{ on }\Gamma(t)\times(0,T),\label{eq:jumpsi-adim-LL_inf-1phase}\\
\b &= 0&&\text{ on }\Gamma(t)\times(0,T),\label{eq:jumpb-adim-LL_inf-1phase}
\end{align}
\end{subequations}
 where the dimensionless domain and boundary are
\begin{align*}
\Omega _T\pare{t}&=\set{ x \in\mathbb{R}^2,\, x_1\in \mathbb{S}^1,\, -\infty< x_2<\eps h(x_1,t)},
\\
\Ga\pare{t}&=\set{ x \in\mathbb{R}^2,\, x_1\in \mathbb{S}^1,\, x_2=\eps h(x_1,t) }.
\end{align*}

\section{The free boundary problem in Arbitrary Lagrangian-Eulerian coordinates}\label{sec-h}
In order to fix the domain, we now pass to the ALE formulation. The reference domain (see figure \ref{Fig2}) and boundary are
$$
\Omega_T=\set{ x \in\mathbb{R}^2, x_1\in \mathbb{S}^1,\;-\infty <x_2<0},
$$
and
$$
\Gamma=\set{ x \in\mathbb{R}^2, x_1\in \mathbb{S}^1,\;x_2=0}.
$$
\begin{figure}[H]
\centering
\begin{tikzpicture}[domain=-pi:pi, scale=0.9]

\shade[top color=luh-dark-blue!40, bottom color=white] plot[domain=-pi:pi] (\x,{1.2}) -- plot[domain=pi:-pi] (\x,{-1});

\draw[ultra thick, smooth, variable=\x, luh-dark-blue] plot (\x,{1.2});

\node[right] at (pi,1.2) {$\Gamma$};

\node at (-0,0.1) {$\Omega_T$};

\node at (0,2) {\t{Vacuum}};

\end{tikzpicture}
\caption{The reference domain.}
\label{Fig2}
\end{figure}
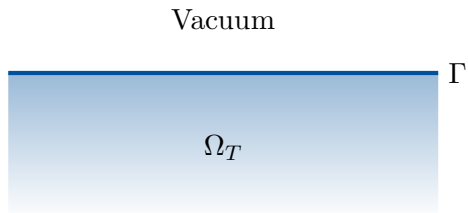
As our current goal is transforming the free boundary problem \eqref{eq:pb-adim-LL_inf-1phase} into a fixed boundary one, we introduce the time dependent diffeomorphism
\[
 \psi (\cdot,t):\Omega_T\to\Omega_T(t)
\]
defined as
\[
 \psi (x_1,x_2,t)=(x_1,x_2+\eps h(x_1,t)).
\]
Then, we are going to compose the equations in problem \eqref{eq:pb-adim-LL_inf-1phase} with $ \psi $ to find its equivalent formulation over $\Omega_T$.
Let us define the concentrations of nutrient,  inhibitor and  pressure in the fixed boundary frame as
\[
S=\si\circ \psi ,\qq B=\b\circ \psi ,\qq
 P=p\circ \psi .
\]
In order to write \eqref{eq:pb-adim-LL_inf-1phase} in terms of $S$, $B$, and $P$, we will need
\begin{equation*}
\N  \psi =\begin{pmatrix}
1&0\\\eps h,_1&1
\end{pmatrix},\qq
 \psi _t=\eps
\begin{pmatrix}
0\\\pat h
\end{pmatrix}
\qq\t{and}\qq
A= \pare{\N  \psi }^{-1}=
\begin{pmatrix}
1&0\\
-\eps h,_1&1
\end{pmatrix}.
\end{equation*}

Since we are interested in the asymptotic behavior in $\eps$, we split these equivalent equations in $S$, $B$, and $P$ as follows: we collect all the terms which depend on $\eps$ in the r.h.s., leaving the rest to the left.\\

In the following computations, we will need to compute
\[
\pat f\circ  \psi \q\t{and}\q \D f\circ \psi
\]
in terms of $F=f\circ \psi$. We use the chain rule  obtaining that
\begin{align*}
\pat f\circ  \psi &=\pat F-A_i^kF,_k \pat \psi ^i
=\pat F-\eps F,_2\pat h,
\\
 (\D f)\circ  \psi &=A_j^i\pare{A_j^kF,_k^T},_i
= \D F
+\eps \bm{R}_{1}[h]\pare{F}+\eps^2  \bm{R}_{2}[h]\pare{F},
\end{align*}
where we have used Einstein convention for summation over repeated indexes, and
\begin{align}
 \bm{R}_{1}[\ell(x_1,t)]\pare{F(x,t)}&=-\pare{\ell,_{11}(x_1,t)F(x,t)+2\ell,_1(x_1,t) F,_{1}(x,t)},_2,\label{eq:Rinf}\\
 \bm{R}_{2}[\ell(x_1,t)]\pare{F(x,t)}&=\ell,_1^2(x_1,t) F,_{22}(x,t).\nonumber
\end{align}
Reasoning in this way, we obtain the system
\begin{subequations}\label{eq:pb-ALE-LL_inf-1phase}
\begin{align}
\pat B+N_1B-M_1&=\eps\pare{B,_2\pat h+\alpha \D B}&&\nonumber\\
&\q+\eps^2\alpha     \bm{R}_{1}[h]\pare{B} &&\text{ in }\Omega_T\times(0,T),\label{eq:pb-ALE-LL_inf-1phase-patB}\\
\pat S+N_2S+N_3B-M_2&=\eps\pare{S,_2\pat h+\D S}&&\nonumber\\
&\q+\eps^2   \bm{R}_{1}[h]\pare{S} &&\text{ in }\Omega_T\times(0,T),\label{eq:pb-ALE-LL_inf-1phase-patS}\\
-\D P+\theta \D S-\rho \pare{S-\tau B}-\omega&= \eps\pare{  \bm{R}_{1}[h](P)-\theta  \bm{R}_{1}[h](S)}&&\nonumber\\
&\q+\eps^2\pare{  \bm{R}_{2}[h](P)-\theta   \bm{R}_{2}[h](P)}
&&\text{ in }\Omega_T\times(0,T),\label{eq:pb-ALE-LL_inf-1phase-DP}\\
P&= -\eta \frac{ h,_{11}}{\pare{1+\eps^2 h,_1^2}^{3/2}}&& \t{ on }\Gamma\times(0,T),\label{eq:pb-ALE-LL_inf-1phase-bcP}\\
\pat h&=-P,_2+\theta S,_2+\eps h,_1 \pare{P,_1-\theta S,_1}&&\nonumber\\
&\q-\eps^2 h,_1^2\pare{P,_2-\theta S,_2}&&\t{ on } \Ga\times(0,T),\label{eq:pb-ALE-LL_inf-1phase-path}\\
S&=0&&\t{ on } \Ga\times(0,T),\label{eq:pb-ALE-LL_inf-1phase-bcS}\\
B&=0&&\t{ on } \Ga\times(0,T).\label{eq:pb-ALE-LL_inf-1phase-bcB}
\end{align}
\end{subequations}

\section{Derivation of the asymptotic model}

\subsection{Methodology and a toy problem}
Now we are going to develop the ALE formulation of the problem in terms of $\eps$. In this way, the nonlinear ALE system will be equivalent to an infinite cascade of linear problems. As a consequence, we can recursively solve the linear problems. Furthermore, these linear problems decouple the equations in the bulk of the tumor colony and the dynamics of the free boundary. The main advantage of this approach is that, as outcome, the elliptic problems can be explicitly solved and substituted in the PDE for the tumor free boundary.

Let us show how our method works in a simpler toy problem. Let us consider the system
$$
\pat a +\eps b\partial_x a=\partial_x^2 a
$$
$$
b-\partial_x^2b=a+\eps b^2.
$$
We introduce the ansatz
$$
a(x,t)=\sum_{n=0}^\infty\eps^n a^{(n)}(x,t),\qq b(x,t)=\sum_{n=0}^\infty\eps^n b^{(n)}(x,t).
$$
The case $n=0$ reduces to
$$
\pat a^{(0)}=\partial_x^2 a^{(0)}
$$
$$
b^{(0)}-\partial_x^2b^{(0)}=a^{(0)}.
$$
This linear system can be explicitly solved using classical Fourier techniques.

The case $n=1$ then is
$$
\pat a^{(1)}+b^{(0)}\partial_x a^{(0)}=\partial_x^2 a^{(1)}
$$
$$
b^{(1)}-\partial_x^2b^{(1)}=a^{(1)}+(b^{(0)})^2,
$$
which can be explicitly solved using Duhamel principle.

If we write the explicit (nonlocal) expressions for $b^{(0)}$ (and $b^{(1)}$) in terms of $a^{(0)}$ and $a^{(1)}$, we will obtain a closed evolution equation for $a^{(1)}$. Indeed
$$
\pat a^{(1)}+G*a^{(0)}\partial_x a^{(0)}=\partial_x^2 a^{(1)},
$$
where $\widehat{G}(n)=\frac{1}{1+n^2}$. We define $A=a^{(0)}+\eps a^{(1)}$ and find
$$
\pat A+\eps G*A\partial_x A=\partial_x^2 A+O(\eps^2).
$$
In this way, we find that $A$, solving a single nonlinear and nonlocal PDE, approaches the original $a$ solving the full parabolic-elliptic toy model.

Even if the idea seems simple, in practice, for the ALE system is a challenging task. The computations for this procedure in the case of the ALE system are long and mostly tedious. As a consequence, we will skip most of them.

\subsection{Ansatz}
We introduce the ansatz
$$
h(x_1,t)=\sum_{n=0}^\infty\eps^n\mh{n}(x_1,t),\qq B(x,t)=\sum_{n=0}^\infty\eps^n \mB{n}(x,t),
$$
$$
S(x,t)=\sum_{n=0}^\infty\eps^n \mS{n}(x,t),\qq P(x,t)=\sum_{n=0}^\infty\eps^n \mP{n}(x,t),
$$
where
$$
B(x,0)=\mB{0}(x,0),\qq
S(x,0)=\mS{0}(x,0),\qq
h(x_1,0)=\mg(x_1,0),
$$
and
$$
\mS{n}(x,0)=\mB{n}(x,0)=0,\qq \mh{n}(x_1,0)=0\qq \t{for }n\ge1.
$$
We consider now the case where
$$
M_1=M_2=\omega=O(\eps^2).
$$
This is, for instance, the case of an avascular tumor with very little access to both nutrients and inhibitors. Similarly, it is the case \red{of vascular tumors} where the amount of nutrients and inhibitors is very similar to the amount consumed by the tumor cells.

\subsection{The case $n=0$}

The first term in the asymptotic series is then
\begin{subequations}\label{eq:pb-ALE-LL_inf-1phase-0}
\begin{align}
\pat \mB{0}+N_1\mB{0}&=0&&\text{ in }\Omega_T\times(0,T),\label{eq:B-ALE-LL_inf-1phase-0}\\
\pat \mS{0}+N_2\mS{0}+N_3\mB{0}&=0&&\text{ in }\Omega_T\times(0,T),\label{eq:S-ALE-LL_inf-1phase-0}\\
-\D \mP{0}+\theta \D \mS{0}-\rho \pare{\mS{0}-\tau \mB{0}}&=0
&&\text{ in }\Omega_T\times(0,T),\label{eq:DP-ALE-LL_inf-1phase-0}\\
\mP{0}&= -\eta \mh{0},_{11}&& \t{ on }\Gamma\times(0,T),\label{eq:bcP-ALE-LL_inf-1phase-0}\\
\pat \mh{0}&=-\mP{0},_2+\theta \mS{0},_2&&\t{ on } \Ga\times(0,T),\label{eq:path-ALE-LL_inf-1phase-0}\\
\mS{0}&=0&&\t{ on } \Ga\times(0,T),\label{eq:bcS-ALE-LL_inf-1phase-0}\\
\mB{0}&=0&&\t{ on } \Ga\times(0,T).\label{eq:bcB-ALE-LL_inf-1phase-0}
\end{align}
\end{subequations}
In order to simplify the notation, we set
\[
\m{B}(x)=\mB{0}(x,0)\qq\t{and}\qq \m{S}(x)=\mS{0}(x,0).
\]

\medskip

We can solve for $B$ and $S$ using, respectively, \eqref{eq:B-ALE-LL_inf-1phase-0} - \eqref{eq:bcB-ALE-LL_inf-1phase-0} and \eqref{eq:S-ALE-LL_inf-1phase-0} - \eqref{eq:bcS-ALE-LL_inf-1phase-0}, finding that
\begin{align}
\mB{0}(x,t)&=e^{-N_1t}\m{B}(x),\label{eq:Bpar0}\\
\mS{0}(x,t)&=e^{-N_2t}\m{L}(x,t),\label{eq:Spar0}
\end{align}
where
\begin{equation}\label{eq:L}
\m{L}(x,t)=\m{S}(x)-\frac{N_3}{N_2-N_1}\pare{e^{(N_2-N_1)t}-1}\m{B}(x).
\end{equation}

\medskip

The expression of the pressure $\mP{0}$ can be found applying \cite[Lemma A.1]{granero2019models} to \eqref{eq:DP-ALE-LL_inf-1phase-0}-\eqref{eq:bcP-ALE-LL_inf-1phase-0}:
\begin{align}
\mP{0}(x,t)&=-\frac{1}{\sqrt{2\pi}}\sum_{k\in\bZ}\biggl[
\frac{1}{2|k|}\pare{\int_{-\infty}^0\w{w}^{(0)}(k,y_2,t)e^{|k|y_2}\,dy_2}\pare{e^{|k|x_2}-e^{-|k|x_2}}\nonumber\\
& \q + \frac{1}{2|k|}\pare{\int_{0}^{x_2}\w{w}^{(0)}(k,y_2,t)\pare{e^{-|k|(x_2-y_2)}-e^{|k|(x_2-y_2)}}\,dy_2} \nonumber\\
& \q - \w{g^{(0)}}(k,t) e^{|k|x_2}
\biggr]e^{ikx_1},\label{eq:P0}
\end{align}
 with $\w{b}=\w{w}^{(0)},\,\w{g}=\w{g^{(0)}}$, and
\begin{align*}
\w{w}^{(0)}&=\theta \w{\D \mS{0}}-\rho \pare{\w{\mS{0}}-\tau \w{\mB{0}}},\\
\w{g^{(0)}}&= \eta|k|^2 \w{\mh{0}}.
\end{align*}

\medskip

We gather the above computations, together with \eqref{eq:path-ALE-LL_inf-1phase-0}, to say that
\begin{align}
\pat \mh{0}(x_1,t)&=-\mP{0},_2(x_1,0,t)+\theta e^{-N_2t}\m{L},_2(x_1,0,t)\nonumber\\
&=-\eta\Lambda^3\mh{0}(x_1,t)+K^{(0)}(x_1,0,t),\label{eq:path01}
\end{align}
where
\begin{align}
K^{(0)}(x_1,0,t)&=-\theta\int_{-\infty}^0 e^{y_2\Lambda} e^{-N_2t}\D\m{L}(x_1,y_2,t)dy_2\nonumber\\
&\q+\rho\int_{-\infty}^0e^{y_2\Lambda} e^{-N_2t}\left(\m{L}(x_1,y_2,t)-\tau e^{(N_2-N_1)t}\m{B} (x_1,y_2)
 \right)dy_2\nonumber\\
&\q +\theta e^{-N_2t}\m{L},_2(x_1,0,t)\label{eq:H0}.
\end{align}

\subsection{The case $n=1$}

We now consider the case $n=1$.

\subsubsection*{The expressions of $\mB{1}$ and $\mS{1}$}

The equations  \eqref{eq:pb-ALE-LL_inf-1phase-patB} - \eqref{eq:pb-ALE-LL_inf-1phase-bcB} for the second term of the asymptotic expansion reads
\begin{align*}
\pat \mB{1}+N_1\mB{1}&=\mB{0},_2\pat \mh{0}+\alpha\D\mB{0}&&\nonumber\\
&=e^{-N_1t}\pare{\m{B},_2\pat \mh{0}+\alpha \D \m{B}} &&\t{in }\Omega_T\times (0,T),\\
\mB{1}&=0&&\t{on }\Gamma\times (0,T),
\end{align*}
thanks also to the expression of $\mB{0}(x,t)$ in \eqref{eq:Bpar0}. This means that
\begin{align*}
\mB{1}(x,t)&=e^{-N_1t}\bra{\m{B},_2(x)\pare{ \mh{0}(x_1,t)- \mg(x_1)} +\alpha t\D \m{B}(x)} .
\end{align*}
We now set
\begin{equation}\label{eq:Qal}
\bm{Q}_\al[\ell(x_1,t)](f(x,t))=\pare{ \ell(x_1,t)- \mg(x_1)}f,_2(x,t)+\al t\D f(x,t),
\end{equation}
and rewrite $\mB{1}$ as
\begin{equation}\label{eq:B1}
\mB{1}(x,t)=e^{-N_1t}\bm{Q}_\al[\mh{0}(x_1,t)](\m{B}(x)).
\end{equation}
Note that, even if $\m{B}$ does not depend on $t$, $\bm{Q}_\al[\mh{0}(x_1,t)](\cdot)$ always depends on time.

\medskip

As far as \eqref{eq:pb-ALE-LL_inf-1phase-patS} - \eqref{eq:pb-ALE-LL_inf-1phase-bcS} is concerned, the computations above and \eqref{eq:Spar0} lead to
\begin{align*}
&\pat \mS{1}(x,t)+N_2\mS{1}(x,t)\\
&=-N_3\mB{1}(x,t)+\mS{0},_2(x,t)\pat \mh{0}(x_1,t)+\D \mS{0}(x,t)\\
&=-N_3e^{-N_1t}\bra{\m{B},_2(x)\pare{ \mh{0}(x_1,t)- \mg(x_1)} +\alpha t\D \m{B}(x)}\\
&\q+e^{-N_2t}\pare{\m{S},_2(x)-\frac{N_3}{N_2-N_1}\pare{e^{(N_2-N_1)t}-1}\m{B},_2(x)}\pat \mh{0}(x_1,t)\\
&\q +e^{-N_2t} \pare{\D \m{S}(x)-\frac{N_3}{N_2-N_1}\pare{e^{(N_2-N_1)t}-1}\D \m{B}(x)}
\qq\qq\t{in }\Omega_T\times (0,T),\\
&\mS{1}=0\qq\qq\qq\qq\qq\qq\qq\qq\qq\qq\qq\qq\q\,\,\t{on }\Gamma\times (0,T),
\end{align*}
so that $\mS{1}$ has the form
\begin{align*}
\mS{1}(x,t)
&=-N_3e^{-N_2t}\int_0^t e^{(N_2-N_1)s}\bra{\m{B},_2(x)\pare{ \mh{0}(x_1,s)- \mg(x_1)} +\alpha s\D \m{B}(x)}\,ds\\
&\q +e^{-N_2t}
\m{S},_2(x)\pare{ \mh{0}(x_1,t)- \mg(x_1)}
\\
&\q -\frac{N_3}{N_2-N_1}e^{-N_2t}
\m{B},_2(x)\int_0^t\pare{e^{(N_2-N_1)s}-1}\pa_s\mh{0}(x_1,s)\,ds
\\
&\q +e^{-N_2t}\bra{t\D \m{S}(x)-\frac{N_3}{N_2-N_1}\D \m{B}(x)\int_0^t e^{(N_2-N_1)s}-1\,ds}.
\end{align*}
We set
\begin{equation}\label{eq:M}
M(t)=\frac{N_3(1-\al)}{N_2-N_1}e^{-N_2t}
\bra{
t e^{(N_2-N_1)t}-\frac{1}{N_2-N_1}
\pare{e^{(N_2-N_1)t}-1}},
\end{equation}
and recall the expression of $\mS{1}$ in \eqref{eq:Qal} so that we can, after computing explicitly the previous integrals, rewrite $\mS{1}$ in compact form as
\begin{align}
\mS{1}(x,t)
&=e^{-N_2t}
\bm{Q}_1[\mh{0}(x_1,t)](\m{L}(x,t))+M(t)\D \m{B}(x)
 .\label{eq:S1}
\end{align}

\subsubsection*{The expression of $\mP{1}$}

We now focus on $\mP{1}$. Equations \eqref{eq:pb-ALE-LL_inf-1phase-DP} -- \eqref{eq:pb-ALE-LL_inf-1phase-bcP} for the second term of the asymptotic expansion read
\begin{align*}
\D \mP{1}&=\theta\D \mS{1}-\rho\pare{\mS{1}-\tau\mB{1}}-  \bm{R}_{1}[\mh{0}](\mP{0})&&\\
&\q+\theta  \bm{R}_{1}[\mh{0}](\mS{0})&&\text{ in }\Omega_T\times(0,T),\\
\mP{1}&= \eta \mh{1},_{11}&& \t{ on }\Gamma\times(0,T),
\end{align*}
where $  \bm{R}_{1}[\mh{0}]$ is given by \eqref{eq:Rinf}.\\
To find $\mP{1}$, we apply  again \cite[Lemma A.1]{granero2019models} with   $\w{b}=\w{w_{(1)}},\,\w{g}=\w{g_{(1)}}$, for
\begin{align*}
\w{w_{(1)}}&=\theta \w{\D \mS{1}}-\rho \pare{\w{\mS{1}}-\tau \w{\mB{1}}}-\w{  \bm{R}_{1}[\mh{0}]}(\mP{0})+\w{  \bm{R}_{1}[\mh{0}]}(\mS{0}),\\
\w{g_{(1)}}&= \eta|k|^2 \w{\mh{1}},
\end{align*}
 obtaining that
\begin{align*}
\mP{1}(x,t)&=-\frac{1}{\sqrt{2\pi}}\sum_{k\in\bZ}\left[
\frac{1}{2|k|}\pare{\int_{-\infty}^0\w{w_{(1)}}(k,y_2,t)e^{|k|y_2}\,dy_2}\pare{e^{|k|x_2}-e^{-|k|x_2}}\right.\nonumber\\
&\left. \q + \frac{1}{2|k|}\pare{\int_{0}^{x_2}\w{w_{(1)}}(k,y_2,t)\pare{e^{-|k|(x_2-y_2)}-e^{|k|(x_2-y_2)}}\,dy_2}\right.\nonumber\\
&\left.\q - \w{g_{(1)}}(k,t) e^{|k|x_2}
\right]e^{ikx_1}.
\end{align*}

\subsubsection*{The evolution equation of $\pat\mh{1}$}

Finally, the second term of the asymptotic expansion of  \eqref{eq:pb-ALE-LL_inf-1phase-path} is
\begin{align*}
\pat \mh{1} &=-\mP{1},_2+\theta \mS{1},_2+\mh{0},_1\pare{\mP{0},_1-\theta\mS{0},_1}&&\t{on }\Gamma\times(0,T).
\end{align*}
We want to compute the explicit expression of $\pat\mh{1}$ in terms of $\m{B}$ and $\m{S}$. To this aim, we have to develop each term in the above equation. \\
We begin focusing on $\mP{1},_2$.
\cite[Lemma A.1]{granero2019models} gives us
\begin{align*}
\mP{1},_2&=\eta\Lambda^3\mh{1}+F_1+F_2+F_3+F_4+F_5&&\t{on }\Gamma\times(0,T),
\end{align*}
where
\begin{align*}
F_1(x_1,0,t)&=\theta\int_{-\infty}^0 e^{y_2\Lambda}\D \mS{1}(x_1,y_2,t)\,dy_2,\\
F_2(x_1,0,t)&=-\rho\int_{-\infty}^0 e^{y_2\Lambda}   \mS{1}(x_1,y_2,t)\,dy_2,
\\
F_3(x_1,0,t)&=\rho\tau\int_{-\infty}^0 e^{y_2\Lambda}  \mB{1}(x_1,y_2,t) \,dy_2,\\
F_4(x_1,0,t)&=-\int_{-\infty}^0 e^{y_2\Lambda}   \bm{R}_{1}[\mh{0}](\mP{0}(x_1,y_2,t))\,dy_2,
\\
F_5(x_1,0,t)&=\int_{-\infty}^0 e^{y_2\Lambda}   \bm{R}_{1}[\mh{0}](\mS{0}(x_1,y_2,t))\,dy_2.
\end{align*}

After a careful and long computation in the Fourier variables, we find that, setting
\begin{equation}\label{eq:K1}
K^{(1)}(x_1,0,t)= -I(x_1,0,t)-K_1^{(1)}(x_1,0,t)+K_2^{(1)}(x_1,0,t),
\end{equation}
\begin{equation}\label{eq:I}
I(x_1,0,t)=\frac{1}{\sqrt{2\pi}}\sum_{k\in\ZZ} e^{ikx_1}\pare{ \w{I}_1(k,0,t)+ \w{I}_2(k,0,t)},
\end{equation}
\begin{align}
 \w{I}_1(k,0,t)
&=-
\frac{1}{2}\sum_{m\in\ZZ} \pare{\int_{-\infty}^{0}\w{w}^{(0)}(m,y_2,t)e^{|m|y_2}\,dy_2}\times\nonumber\\
&\qquad\times\int_{-\infty}^0 e^{y_2|k|} (k-m)^2\w{\mh{0}}(k-m)
\pare{e^{|m|y_2}+e^{-|m|y_2}}
\,dy_2\nonumber
\\
&  \q +\frac{1}{2} \sum_{m\in\ZZ}\int_{-\infty}^0 e^{y_2|k|} (k-m)^2\w{\mh{0}}(k-m) \times\nonumber\\
&\qquad\times
 \pare{\int_{0}^{y_2}\w{w}^{(0)}(m,y_2',t)\pare{e^{-|m|(y_2-y_2')}+e^{|m|(y_2-y_2')}}\,dy_2'}
\,dy_2,\label{eq:I1k}
\\
 \w{I}_2(k,0,t)
&=\sum_{m\in\ZZ} \pare{\int_{-\infty}^{0}\w{w}^{(0)}(m,y_2,t)e^{|m|y_2}\,dy_2}\times\nonumber\\
&\qquad\times\int_{-\infty}^0 e^{y_2|k|}
(k-m)\w{\mh{0}}
(k-m) \pare{e^{|m|y_2}+e^{-|m|y_2}}
\,dy_2\nonumber
\\
&\q-\sum_{m\in\ZZ}\int_{-\infty}^0 e^{y_2|k|}
(k-m)\w{\mh{0}}
(k-m)\times\nonumber\\
&\qquad\times\pare{\int_{0}^{y_2}\w{w}^{(0)}(m,y_2',t)\pare{e^{-|m|(y_2-y_2')}+e^{|m|(y_2-y_2')}}\,dy_2'}
\,dy_2,\label{eq:I2k}
\end{align}
\begin{align*}
\w{w}^{(0)}&=\theta\widehat{\D\mS{0}}-\rho\pare{\w{\mS{0}}-\tau\w{\mB{0}}},
\end{align*}
\begin{equation}\label{eq:K11}
K_1^{(1)}=F_1+F_2+F_3+F_5,
\end{equation}
\begin{equation}\label{eq:K12}
K_2^{(1)}=\theta e^{-N_2t}\bm{Q}_1[\mh{0}](\m{L},_2)+\theta  M(t)\D \m{B},_2-\theta\mh{0},_1e^{-N_2t}\m{L},_1.
\end{equation}
then  $\pat \mh{1}$ reads
\begin{align}\label{eq:path11}
\pat \mh{1}(x_1,t) &=-\eta\Lambda^3\mh{1}(x_1,t)+\eta\pare{\jump{H,\mh{0}(x_1,t)}H\mh{0},_{111}(x_1,t)},_1\nonumber \\
&\q+K^{(1)}(x_1,0,t).
\end{align}

\subsection{The asymptotic model}
We use the equations of $\pat\mh{0}$ in \eqref{eq:path01} and of $\pat\mh{1}$ in \eqref{eq:path11} to write the \red{equation of $\pat h$:}
\begin{align}\label{eq:sum-h}
\pat\mh{0}+\eps\pat\mh{1}&=-\eta\Lambda^3\mh{0}-\eps\eta\Lambda^3\mh{1}+\eps\eta\pare{\jump{H,\mh{0}}H\mh{0},_{111}},_1\nonumber\\
&\q+K^{(0)}-\eps K^{(1)},
\end{align}
for $K^{(0)}$ as in \eqref{eq:H0} and for $K^{(1)}$ as in \eqref{eq:K1}.\\
Now, let
\[
g=\mh{0}+\eps\mh{1}.
\]
Equation \eqref{eq:sum-h} can be written in terms of $g$ as follows
\begin{align}
\pat g(x_1,t)&=-\eta \Lambda^3 g(x_1,t)+\eps\eta\pare{\jump{H,g(x_1,t)}H g,_{111}(x_1,t)},_1\nonumber\\
&\q-\eps\theta g(x_1,t)e^{-N_2t}\m{L},_1(x_1,0,t)+\eps \widetilde{K}^{(1)}(x_1,0,t)\nonumber
\\
&\q
+K^{(0)}(x_1,0,t)
+O(\eps^2).\label{eq:patg}
\end{align}
with
\[
\widetilde{K}^{(1)}=\red{-\widetilde{I}-\widetilde{K}_1^{(1)}}+\widetilde{K}_2^{(1)},
\]
and
\begin{align}
\widetilde{I}(x_1,0,t)&=\frac{1}{\sqrt{2\pi}}\sum_{k\in\ZZ} e^{ikx_1}\pare{\widehat{\widetilde{I}}_1(k,0,t)+\widehat{\widetilde{I}}_2(k,0,t)},
\end{align}
\begin{align}
&\widehat{\widetilde{I}}_1(k,0,t)\nonumber\\
&=-
\frac{1}{2}\sum_{m\in\ZZ}  \pare{\int_{-\infty}^{0}\w{w}^{(0)}(m,y_2,t)e^{|m|y_2}\,dy_2}\times\nonumber\\
&\qquad\times\int_{-\infty}^0 e^{y_2|k|} (k-m)^2\w{g}(k-m)
\pare{e^{|m|y_2}+e^{-|m|y_2}}
\,dy_2 \nonumber
\\
&  \q +\frac{1}{2} \sum_{m\in\ZZ}\int_{-\infty}^0 e^{y_2|k|} (k-m)^2\w{g}(k-m)\times\nonumber\\
&\qquad\times
 \pare{\int_{0}^{y_2}\w{w}^{(0)}(m,y_2',t)\pare{e^{-|m|(y_2-y_2')}+e^{|m|(y_2-y_2')}}\,dy_2'}
\,dy_2, \label{eq:J1}\\
&\widehat{\widetilde{I}}_2(k,0,t)\nonumber\\
&=\sum_{m\in\ZZ}  \pare{\int_{-\infty}^{0}\w{w}^{(0)}(m,y_2,t)e^{|m|y_2}\,dy_2}\times\nonumber\\
&\qquad\times\int_{-\infty}^0 e^{y_2|k|}
(k-m)\w{g}(k-m) \pare{e^{|m|y_2}+e^{-|m|y_2}}
\,dy_2 \nonumber
\\
&\q-\sum_{m\in\ZZ}\int_{-\infty}^0 e^{y_2|k|}
(k-m)\w{g}(k-m)\times\nonumber\\
&\qquad\times\pare{\int_{0}^{y_2}\w{w}^{(0)}(m,y_2',t)\pare{e^{-|m|(y_2-y_2')}+e^{|m|(y_2-y_2')}}\,dy_2'}
\,dy_2,\label{eq:J2}
\end{align}
\begin{align}
\widetilde{K}_1^{(1)}(x_1,0,t)
&=\theta e^{-N_2t}\int_{-\infty}^0 e^{y_2\Lambda}\bm{Q}_{1}[g(x_1,t)](\D \m{L}(x_1,y_2,t))\,dy_2\nonumber\\
&\q+\theta M(t)\int_{-\infty}^0 e^{y_2\Lambda}\D^2 \m{B}(x_1,y_2)
\,dy_2\nonumber\\
&\q-\rho e^{-N_2t}\int_{-\infty}^0 e^{y_2\Lambda}\bm{Q}_{1}[g(x_1,t)]( \m{L}(x_1,y_2,t))\,dy_2\nonumber\\
&\q-\rho M(t)\int_{-\infty}^0 e^{y_2\Lambda}\D \m{B}(x_1,y_2)\nonumber
\\
&\quad+\rho\tau e^{-N_1t}\int_{-\infty}^0 e^{y_2\Lambda}\bm{Q}_{\al}[g(x_1,t)](\m{B}(x_1,y_2)) \,dy_2\nonumber\\
&\q
 -e^{-N_2t}\int_{-\infty}^0 e^{y_2\Lambda}  \bm{R}_{1}[g(x_1,t)]\pare{\m{L}(x_1,y_2,t)}\,dy_2,\label{eq:Fsum}
\end{align}
and
\begin{equation}\label{eq:K12b}
\widetilde{K}_2^{(1)}(x_1,0,t)=\theta e^{-N_2t}\bm{Q}_1[g(x_1,t)](\m{L},_2)+\theta  M(t)\D \m{B},_2-\theta g,_1(x_1,t)e^{-N_2t}\m{L},_1.
\end{equation}

As a consequence, our asymptotic model reads
\begin{align*}
\pat g(x_1,t)&=-\eta \Lambda^3 g(x_1,t)+\eps\eta\pare{\jump{H,g(x_1,t)}H g,_{111}(x_1,t)},_1\nonumber
\\
&\q-\eps\theta g(x_1,t)e^{-N_2t}\m{L},_1(x_1,0,t)+\eps \widetilde{K}^{(1)}(x_1,0,t)
+K^{(0)}(x_1,0,t).
\end{align*}

\section{An explicit particular case}
We focus our attention on the particular case where both the inhibitor and the nutrient depend only on the depth $x_2$. Such hypothesis is reasonable given the assumption of nutrients and inhibitors moving deep in the tumor via diffusion.

To simplify the computations below, we assume that inhibitor and nutrient have the same parameters, and we use the following notation:  $D_i=D_n=D$, $\de_i=\de_n=\de$, and so on. We also set $\tau=1$. In the previous sections, we supposed that $M_1=M_2=\omega=O(\eps^2)$. Then, equations
 \eqref{eq:eps-al}, \eqref{eq:N123}, \eqref{eq:M12}, \eqref{eq:theta-rho-om-eta} become
\begin{equation*}
\eps=\frac{H}{L},\qq \alpha=1,
\end{equation*}
\begin{equation}\label{eq:parameter2-x2}
N_1=N_2=N=LH\frac{\de+\lm}{D},\qq N_3=\ga\frac{LH}{D},
\end{equation}
\begin{equation*}
M_1=0
,\qq  M_2=0,\qq
\theta=\frac{\chi\tsi}{D},\q \rho=\frac{\mu\tsi L^2}{D},\q \omega=0
,\qq \eta=\frac{\nu H}{L^2D}.
\end{equation*}
Note that, with this choice of parameters, we have that the function $M(t)$ defined in \eqref{eq:M} is zero.

If the data do not depend on $x_1$, then the \cite[Lemma A1]{granero2019models} and the computations we made before apply with $k=0$. Indeed, suppose that $b=b(x_2)$ and $g\equiv0$ in \cite[Lemma A1, Eq. (66)]{granero2019models}. Hence, passing to Fourier, the equation $\D u=b$ becomes
\[
-k^2\w{u}(k,x_2)+\w{u},_{22}(k,x_2)=\w{b}(x_2).
\]
The lack of dependence on $k$ of $\w{b}$ implies that $k=0$ in the whole equation, hence $\w{u}$ verifies
\[
\w{u},_{22}(x_2)=\w{b}(x_2).
\]

We consider the following initial data:
\begin{equation*}
\m{B}(x)=c_{\m{B}} e^{x_2}\sin x_2\qquad\t{and}\qquad \m{S}(x)=c_{\m{S}}e^{x_2}\sin x_2
\qquad\t{for }c_{\m{B}},\,c_{\m{S}}>0.
\end{equation*}
We want to simplify \eqref{eq:patg} with this choice. This means that we have to find the expressions of $K^{(0)}$ and $\widetilde{K}^{(1)}$ according to these data. Explicit computations show that for these initial data, the problem reduces to the study of the following PDE
\begin{align}\label{eq:particular}
\pat g(x_1,t)
&=-\eta \Lambda^3 g(x_1,t)+\eps\eta\pare{\jump{H,g(x_1,t)}H g,_{111}(x_1,t)},_1
\nonumber\\
&\q -\eps\left(
2\theta e^{-Nt}\alpha(t)\pare{g(x_1,t)-\mg(x_1)}+e^{-Nt}\alpha(t)(2\theta-\rho)t+\rho e^{-Nt}tc_{\m{B}}\right)\nonumber\\
&\q +\theta\eps e^{-Nt}2\alpha(t)\pare{(g(x_1,t)-\mg(x_1))+t}
 -\eps\frac{\rho}{2}\alpha(t)e^{-Nt}g,_{11}(x_1,t)\nonumber\\
&\q
-\frac{e^{-Nt}}{2}\rho \pare{\alpha(t)-c_{\m{B}}},
\end{align}
being $\mg(x_1)=g(x_1,0)$.

We now prove the local well-posedness of \eqref{eq:particular}. \red{We here recall the definition of Wiener spaces}
$A^s(\TT)$
\begin{equation*}
	A^s(\TT)=\left\{u\in L^1(\TT),\q \|u\|_{A^s(\TT)}:=\sum_{k\in\ZZ} |k|^s|\widehat{u}(k)|<\infty\right\}.
\end{equation*}
\red{In the following, we use the notation
\[
\|f\|_{A^s}  =\|f\|_{A^s(\TT)}.
\]
}
Then, the main result reads

\begin{theorem}[Well-posedness]
Let $\mg\in A^1\red{(\TT)}$ be such that
$$
\norm{\mg}_{A^1}\ll 1.
$$
Then, there exists a time $T$ and a unique solution
$$
g\in C([0,T],A^1\red{(\TT)})\cap L^1(0,T;A^4\red{(\TT)})
$$
of \eqref{eq:particular}.
\end{theorem}
\begin{proof}
\textbf{Sketch of the proof:} The proof follows the ideas in \cite{granero2020asymptotic}. Finally, we need to impose a smallness asumption on the initial data to close the estimates. Furthermore, due to the forcing induced by the nutrients and inhibitors, our estimates allow to a controlled norm growth which eventually imply that the smallness assumption on the initial data does not propagate in time. As a consequence, we have a local in time existence for small initial data.

This goal is achieved in different steps.\\

\textbf{A priori estimate on the mean of $g$:} We compute
\begin{align*}
	\frac{d}{dt}\int_{\mathbb{T}}g(x,t)dx&\leq e^{-Nt}\gamma(t)\left|\int_{\mathbb{T}}g(x,t)dx\right|+\gamma(t)e^{-Nt},
\end{align*}
for certain (explicit) $\gamma(t)$ depending on the initial data and the parameters of the problem that grows linearly in time and may change from line to line.

\textbf{A priori estimate on the derivative of $g$:}
We will make use of the following Poincar\'e type inequalities
\begin{align*}
\|g\|_{A^1}&\leq \av{\int_{\mathbb{T}}g(x,t)dx}+\|g\|_{\dot{A}^1},\\
\|g\|_{\dot{A}^0}&\leq \|g\|_{\dot{A}^1}.
\end{align*}
We have that
\begin{align*}
\sum_{k\in\ZZ}	|k|\pat\av{ \w{g}(k,t)}&=\frac{d}{dt}\norm{g(t)}_{\dot{A}^1},
\\
-\eta\sum_{k\in\ZZ}|k|^4\av{\w{g}(k,t)}&=-\eta\norm{g(t)}_{\dot{A}^4},
\\
2\eta\sum_{k\in\ZZ}\sum_{m\in\ZZ }|k||m||k-m|^3\av{\w{g}(k-m,t)}\av{\w{g}(m,t)}&\le 2\eta\norm{g(t)}_{A^1}\norm{g(t)}_{A^4}.
\end{align*}
Thus,
\begin{align*}
\frac{d}{dt}\norm{g(t)}_{\dot{A}^1}&\le-\eta\pare{1-2\norm{g(t)}_{\dot{A}^1}}\norm{g(t)}_{\dot{A}^4}+e^{-Nt}\pare{\norm{g(t)}_{\dot{A}^3}+\|\mg\|_{\dot{A}^1}+1} \gamma(t)
.
\end{align*}
Using interpolation, we find that
\begin{align*}
\frac{d}{dt}\norm{g(t)}_{\dot{A}^1}&\le-\eta\pare{1-2\norm{g(t)}_{\dot{A}^1}}\norm{g(t)}_{\dot{A}^4}\\
&\q+e^{-Nt}\pare{\norm{g(t)}_{\dot{A}^1}^{1/3}\norm{g(t)}_{\dot{A}^4}^{2/3}+\|\mg\|_{\dot{A}^1}+1} \gamma(t)
,
\end{align*}
and, using Young inequality, we conclude that
\begin{align*}
\frac{d}{dt}\norm{g(t)}_{\dot{A}^1}&\le-\eta\pare{1-\gamma(t)\delta e^{-Nt}-2\norm{g(t)}_{\dot{A}^1}}\norm{g(t)}_{\dot{A}^4}\\
&\q+e^{-Nt}\pare{\frac{1}{\delta}\norm{g(t)}_{\dot{A}^1}+\|\mg\|_{\dot{A}^1}+1} \gamma(t)
,
\end{align*}
for $\delta>0$ artbitrary. Using the smallness of the initial data together with taking a small enough $\delta>0$
we obtain that
\begin{align*}
\frac{d}{dt}\norm{g(t)}_{\dot{A}^1}+c\eta\norm{g(t)}_{\dot{A}^4}&\le
e^{-Nt}\gamma(t)(\|g(t)\|_{\dot{A}^1}+1).
\end{align*}

\textbf{Closing the estimates:} Collecing the previous estimates we find the inequality
\begin{align*}
\frac{d}{dt}\norm{g(t)}_{A^1}&\le
e^{-Nt}\gamma(t)(\|g(t)\|_{A^1}+1).
\end{align*}
From where a uniform time of existence $0<T\ll 1$ can be concluded. This time cannot, in general, be taken arbitrarily large because the smallness condition may not be satisfied for large time intervals.

\textbf{Uniqueness:} The uniqueness follows from a standard contradiction argument together with the parabolic gain of regularity.

\end{proof}

\section{Discussion}
In this paper we have developed a new nonlinear and nonlocal PDE modelling the dynamics of a multilayer tumor colony in an early stage (where there is no necrotic inner region) \cite{kim2004three,mueller1997three} under one of the following set of hypotheses:
\begin{enumerate}
\item in the case of an avascular tumor with very little access to both nutrients and inhibitors
\item or in the case \red{of vascular tumors} where the amount of nutrients and inhibitors is very similar to the amount consumed by the tumor cells.
\end{enumerate}
Remarkably, if the stantard approach based in conservation laws and heuristic chemical balances leads to systems of nonlinear and nonlocal PDEs with a free boundary in several space dimensions, our approach arrives to a single nonlinear and nonlocal PDE capturing the whole dynamics up to an explicit error term. In the simplest case of depth depending nutrient and inhibitir, our PDE reads
\begin{align*}
&\pat g(x_1,t)\nonumber\\
&=-\eta \Lambda^3 g(x_1,t)+\eps\eta\pare{\jump{H,g(x_1,t)}H g,_{111}(x_1,t)},_1
\nonumber\\
&\q -\eps\left(
2\theta e^{-Nt}\alpha(t)\pare{g(x_1,t)-\mg(x_1)}+e^{-Nt}\alpha(t)(2\theta-\rho)t+\rho e^{-Nt}tc_{\m{B}}\right)\nonumber\\
&\q +\theta\eps e^{-Nt}2\alpha(t)\pare{(g(x_1,t)-\mg(x_1))+t}
 -\eps\frac{\rho}{2}\alpha(t)e^{-Nt}g,_{11}(x_1,t)\\
&\q
-\frac{e^{-Nt}}{2}\rho \pare{\alpha(t)-c_{\m{B}}}.
\end{align*}
Even if rather intrincate, the previous PDE opens the door to a more detailed mathematical analysis of the tumor growth dynamics. A first step in this direction is the well-posedness result that we have proved in this paper. Moreover, another advantage of our approach appears when trying to simulate the dynamics of the tumor colony. Indeed, instead of having to code a two-dimensional free boundary evolution equation, we have to simulate a single one-dimensional PDE with periodic boundary conditions. This later task could be achieved with a spectral Fourier collocation method and it should probably done in a future work.

\section*{Acknowledgment}
R.G-B was supported by the project "Mathematical Analysis of Fluids and Applications" Grant PID2019-109348GA-I00 funded by MCIN/AEI/ 10.13039/501100011033 and acronym "MAFyA". This publication is part of the project PID2019-109348GA-I00 funded by MCIN/ AEI /10.13039/501100011033. R.G-B is also supported by a 2021 Leonardo Grant for Researchers and Cultural Creators, BBVA Foundation. The BBVA Foundation accepts no responsibility for the opinions, statements, and contents included in the project and/or the results thereof, which are entirely the responsibility of the authors.\\
 The work of M.M. was partially supported by a public grant as part of the Investissement d'avenir project, FMJH, and by Grant RYC2021-033698-I, funded by the Ministry of Science and Innovation/State Research Agency/10.13039/501100011033 and by the European Union "NextGenerationEU/Recovery, Transformation and Resilience Plan". Both authors are funded by  the project "An\'alisis Matem\'atico Aplicado y Ecuaciones Diferenciales" Grant PID2022-141187NB-I00 funded by MCIN /AEI /10.13039/501100011033 / FEDER, UE and acronym "AMAED". This publication is part of the project PID2022-140494NA-I00 funded by MCIN/ AEI /10.13039/501100011033.


\end{document}